\documentclass[a4paper,12pt,reqno]{amsproc}
\usepackage[shortlabels]{enumitem}
\usepackage{mathrsfs}
\usepackage{graphicx}

\newtheorem{lemma}{Lemma}
\newtheorem{definition}{Definition}
\newtheorem{theorem}{Theorem}
\newtheorem{example}{Example}
\newtheorem{remark}{Remark}
\newtheorem{corollary}{Corollary}
\newtheorem*{oprobs}{Open Problems}

\newcommand{\vek}[1]{\mathbf{#1}}
\newcommand{\mat}[1]{\mathbf{#1}}
\newcommand{\gauss}[3]{\genfrac{[}{]}{0pt}{}{#1}{#2}_{#3}}
\newcommand{\gaussm}[2]{\genfrac{[}{]}{0pt}{}{#1}{#2}}
\newcommand{\subspaces}{\mathrm{L}}
\newcommand{\smax}{\mathrm{A}}
\DeclareMathOperator{\PG}{PG}
\DeclareMathOperator{\AG}{AG}
\DeclareMathOperator{\Mat}{M}
\DeclareMathOperator{\Aut}{Aut}
\DeclareMathOperator{\GL}{GL}
\DeclareMathOperator{\Hom}{Hom}
\DeclareMathOperator{\End}{End}
\DeclareMathOperator{\rank}{rk}
\DeclareMathOperator{\kernel}{Ker}

\newcommand{\LAG}{\AG_{\mathrm{\ell}}}
\newcommand{\imat}{\mathbf{I}}
\newcommand{\uvek}{\mathbf{e}}
\newcommand{\trace}{\mathrm{Tr}}
\newcommand{\opp}{\circ}

\newcommand{\dham}{\mathrm{d}_{\mathrm{Ham}}}
\newcommand{\rdist}{\mathrm{d}_{\mathrm{r}}}
\newcommand{\sdist}{\mathrm{d}_{\mathrm{s}}}
\newcommand{\F}{\mathbb{F}}
\DeclareMathOperator{\cm}{cm}
\newcommand{\cms}{\mathcal{K}}
\newcommand{\tp}{\mathsf{T}}
\newcommand{\attspace}{\mathscr{H}}

\title[Optimal Binary $(6,77,4;3)$ Subspace Codes]{Optimal Binary
  Subspace Codes\\of Length $6$, Constant Dimension $3$\\and Minimum Subspace
  Distance $4$}

\author{Thomas Honold}
\address{Thomas Honold,
  Department of Information and Electronic Engineering,
  Zhejiang University, 
  38 Zheda Road,
  310027 Hangzhou, 
  China}
\email{honold@zju.edu.cn}
\thanks{The work of the first author
was supported by the National Basic Research Program of China 
(973) under Grant No. 2009CB320903. The work of the two latter authors 
was supported by the ICT COST Action IC1104.}

\author{Michael Kiermaier}
\address{Michael Kiermaier,
  Mathematisches Institut,
  Universit\"at Bayreuth,
  D-95440 Bayreuth,
  Germany} 
\email{michael.kiermaier@uni-bayreuth.de} 

\author{Sascha Kurz}
\address{Sascha Kurz,
  Mathematisches Institut,
  Universit\"at Bayreuth,
  D-95440 Bayreuth,
  Germany} 
\email{sascha.kurz@uni-bayreuth.de} 

\begin{document}



\date{}

\dedicatory{In memoriam Axel Kohnert (1962--2013)}

\begin{abstract}
  It is shown that the maximum size of a binary subspace code of
  packet length $v=6$, minimum subspace distance $d=4$, and constant
  dimension $k=3$ is $M=77$; in Finite Geometry terms, the maximum
  number of planes in $\PG(5,2)$ mutually intersecting in at most a
  point is $77$.  Optimal binary $(v,M,d;k)=(6,77,4;3)$ subspace codes
  are classified into $5$ isomorphism types, and a computer-free
  construction of one isomorphism type is provided.  The construction
  uses both geometry and finite fields theory and generalizes to any
  $q$, yielding a new family of $q$-ary $(6,q^6+2q^2+2q+1,4;3)$ subspace
  codes.
\end{abstract}

\maketitle

\section{Introduction}\label{sec:intro}
Let $q>1$ be a prime power, $v\geq 1$ an integer, and $V\cong\F_q^v$ a
$v$-dimensional vector space over $\F_q$. The set
$\subspaces(V)$ of all subspaces of $V$, or flats of the
projective geometry $\PG(V)\cong\PG(v-1,q)$, forms a metric space
with respect to the \emph{subspace distance} defined by $\sdist(U,U')
=\dim(U+U')-\dim(U\cap U')$. The so-called \emph{Main Problem of Subspace
  Coding} asks for the determination of the maximum sizes of codes in
the spaces $(\subspaces(V),\sdist)$ (so-called \emph{subspace codes})
with given minimum distance and the classification of the
corresponding optimal codes. Since $V\cong\F_q^v$ induces an isometry
$\left(\subspaces(V),\sdist\right)\cong(\subspaces(\F_q^v),\sdist)$,
the particular choice of the ambient vector space $V$ does not matter
here.

The metric space $(\subspaces(V),\sdist)$ may be viewed as a $q$-analogue 
of the Hamming space $(\F_2^v,\dham)$ used in conventional coding
theory via
the subset-subspace analogy~\cite{knuth71a}.\footnote{If $\F_2^v$ is
  identified with the set of subsets of $\{1,\dots,v\}$ in the usual
  way, then
  $\dham(\vek{c},\vek{c}')=\#(\vek{c}\cup\vek{c}')-\#(\vek{c}\cap\vek{c}')$
  for $\vek{c},\vek{c}'\in\F_2^v$.} The corresponding main problem
of conventional coding theory has been around for several decades
and is well-studied by now; cf.\ the
extensive treatise \cite{pless-huffman98a,pless-huffman98b}, for
example. 
Whereas the classical main problem resulted from Shannon's description of
point-to-point channel coding,
the main problem of subspace coding has emerged only recently in
connection with the K\"otter-Kschischang model of noncoherent network
coding; cf.\ \cite{koetter-kschischang08} and the survey
\cite{kschischang12}. A recent survey 
on the main problem of subspace coding can be found in
reference \cite{etzion13b}, to which we also refer for more background
on this problem. However, it is only fair to say that its surface has only
been scratched.

Our contribution to the main problem of subspace coding is the
re\-so\-lu\-tion of the smallest hitherto open constant-dimension
case---binary subspace codes of packet length $v=6$ and constant
dimension $3$. This answers a question posed in \cite{etzion13b}
(Research Problem~11, cf.\ also the first table following Research
Problem~10). It also forms the major step towards the solution of
the main problem for the smallest open ``mixed-dimension'' case
$(\subspaces(\F_2^6),\sdist)$, which is left for a future publication.

In order to state our results, we make the
following fundamental
\begin{definition}
  \label{dfn:scode}
  A \emph{$q$-ary $(v,M,d)$ subspace code} (or \emph{subspace code
  with parameters $(v,M,d)_q$}) is a set $\mathcal{C}$ of subspaces of
$V\cong\F_q^v$ with size $\#\mathcal{C}=M\geq 2$ and minimum subspace
distance
\[
    \sdist(\mathcal{C})=\min\bigl\{\sdist(U,U');U,U'\in\mathcal{C},U\neq
U'\bigr\}=d\text{.}
\]
If all subspaces in $\mathcal{C}$ have the same
dimension $k\in\{1,\dots,v-1\}$, then $\mathcal{C}$ is said to be a
\emph{$q$-ary $(v,M,d;k)$ constant-dimension subspace
  code}.\footnote{Here the alternative notation $(v,M,d;k)_q$ 
  also applies.} 
The maximum size of a $q$-ary $(v,M,d)$ subspace code (a $q$-ary
$(v,M,d;k)$ constant-dimension subspace code) is denoted by
$\smax_q(v,d)$ (respectively, $\smax_q(v,d;k)$).
\end{definition}
In what follows, except for the conclusion part, we will restrict
ourselves to constant-dimension codes and the numbers
$\smax_q(v,d;k)$. For $0\leq k\leq v=\dim(V)$ we write
$\gaussm{V}{k}=\bigl\{U\in\subspaces(V);\dim(U)=k\bigr\}$ (the set of
$(k-1)$-flats of $\PG(V)\cong\PG(v-1,q)$). Since
$\sdist(U,U')=2k-2\dim(U\cap U')=2\dim(U+U')-2k$ for
$U,U'\in\gaussm{V}{k}$, the minimum distance of every
constant-dimension code $\mathcal{C}\subseteq\gaussm{V}{k}$ is an even
integer $d=2\delta$, and $\delta$ is characterized by
\[
    k-\delta=\max\left\{\dim(U\cap U');U,U'\in\mathcal{C},U\neq U'\right\}\text{.}
\]

The numbers $\smax_q(v,d;k)$ are defined only for $1\leq k\leq v-1$
and even integers $d=2\delta$ with $1\leq \delta\leq\min\{k,v-k\}$. Moreover,
since sending a subspace $U\in\gaussm{V}{k}$ to its orthogonal space
$U^\perp$ (relative to a fixed non-degenerate symmetric bilinear form
on $V$) induces an isomorphism (isometry) of metric spaces
$\left(\gaussm{V}{k},\sdist\right)\to\left(\gaussm{V}{v-k},\sdist\right)$,
$\mathcal{C}\mapsto\mathcal{C}^\perp=\{U^\perp;U\in\mathcal{C}\}$, we have
$\smax_q(v,d;k)=\smax_q(v,d;v-k)$. Hence it suffices to determine the
numbers $\smax_q(v,d;k)$ for $1\leq k\leq v/2$ and
$d\in\{2,4,\dots,2k\}$.

The main result of our present work is
\begin{theorem}
  \label{thm:main}
  $\smax_2(6,4;3)=77$, and there exist exactly five isomorphism
  classes of optimal binary $(6,77,4;3)$ constant-dimension subspace codes.
\end{theorem}
In the language of Finite Geometry, Theorem~\ref{thm:main} says that
the maximum number of planes in $\PG(5,2)$ intersecting each other
in at most a point is $77$, with five optimal solutions up to
geometric equivalence. Theorem~\ref{thm:main} improves on the
previously known inequality $77\leq\smax_2(6,4;3)\leq 81$, the lower
bound being due to a computer construction of a binary $(6,77,4;3)$
subspace code in \cite{kohnert-kurz08}.\footnote{For the (rather
  elementary) upper bound see Lemma~\ref{lma:bound6}.}

The remaining numbers $\smax_2(6,d;k)$ are known and easy to find. A
complete list is given in Table~\ref{tbl:smax_2(6,d;k)}, the new entry
being indicated in bold
type.\footnote{$\smax_2(6,2k;k)=\frac{2^6-1}{2^k-1}$ is equal to the
  size of a $(k-1)$-spread (partition of the point set into $(k-1)$-flats)
  of $\PG(5,2)$.}

\begin{table}[htbp]
  \centering
  \(
  \begin{array}{c|ccc}
    k\backslash d&2&4&6\\\hline
    1&63\\
    2&651&21\\
    3&1395&\mathbf{77}&9\\
  \end{array}
  \)
  \caption{The numbers $\smax_2(6,d;k)\label{tbl:smax_2(6,d;k)}$} 
\end{table}

Theorem~\ref{thm:main} was originally obtained by an extensive
computer search, which is described in Section~\ref{sec:comp}.
Subsequently, inspired by an analysis of the data on the five extremal
codes provided by the search (cf.\ Section~\ref{ssec:analysis}) and
using further geometric ideas, we were able to produce a computer-free
construction of one type of extremal code. This is described in
Section~\ref{sec:nocomp}. It remains valid for all prime powers $q>2$,
proving the existence of $q$-ary $(6,q^6+2q^2+2q+1,4;3)$ subspace
codes for all $q$. This improves the hitherto best known construction
with the parameters $(6,4,q^6 + q^2 + 1; 3)$ in
\cite[Ex.~1.4]{trautmann-rosenthal10} by $q^2 + 2q$ codewords.  As a
consequence, we have
\begin{theorem}
  \label{thm:allq}
  $\smax_q(6,4;3)\geq q^6+2q^2+2q+1$ for all prime powers $q\geq 3$.
\end{theorem}

The ideas and methods employed in
Section~\ref{sec:nocomp} may
also prove useful, as we think, for the construction of good
constant-dimension subspace codes with other parameter sets. Since
they circumvent the size restriction imposed on constant-dimension
codes containing a lifted MRD code, they can
be seen as a partial answer to Research Problem~2 in \cite{etzion13b}.

\section{Preparations}\label{sec:prelim}
\subsection{The Recursive Upper Bound and Partial
  Spreads}\label{ssec:bound}
In the introduction we have seen that a $(v,M,d;k)_q$
constant-dimension subspace code is the same as a set $\mathcal{C}$ of
($k-1$)-flats in $\PG(v-1,q)$ with $\#\mathcal{C}=M$ and the
following property: $t=k-d/2+1$ is the smallest integer such that
every ($t-1$)-flat of $\PG(v-1,q)$ is contained in at most one
($k-1$)-flat of $\mathcal{C}$. This property (and $t\geq 2$) implies
that for any ($t-2$)-flat $F$ of $\PG(v-1,q)$ the ``derived subspace code''
$\mathcal{C}_F=\{U\in\mathcal{C};U\supseteq F\}$ forms a partial
($k-t$)-spread\footnote{A partial $t$-spread is a set of $t$-flats which are
  pairwise disjoint when viewed as point sets.} in the quotient geometry
$\PG(v-1,q)/F\cong\PG(v-t,q)$. The maximum size of such a partial
spread is $\smax_q\bigl(v-t+1,2(k-t+1);k-t+1\bigr)
=\smax_q(v-k+d/2,d;d/2)$. Counting the pairs $(F,U)$ with a ($t-2$)-flat
$F$ and a ($k-1$)-flat $U\in\mathcal{C}$ containing $F$ in two ways,
we obtain the following bound, which can also be easily derived by iterating
the Johnson type bound II in \cite[Th.~3]{xia-fangwei09} or
\cite[Th.~4]{etzion-vardy11a}.
\begin{lemma}
  \label{lma:bound}
  The maximum size $M=\smax_q(v,d;k)$ of a $(v,M,d;k)_q$ subspace code
  satisfies the upper bound
  \begin{equation}
    \label{eq:bound}
    \smax_q(v,d;k)\leq\frac{\gauss{v}{t-1}{q}}{\gauss{k}{t-1}{q}}
    \cdot\smax_q\bigl(v-k+d/2,d;d/2\bigr),
  \end{equation}
  and the second factor on the right of \eqref{eq:bound} is equal to
  the maximum size of a partial ($d/2-1$)-spread in $\PG(v-k+d/2-1,q)$.
\end{lemma}
The numbers $\smax_q(v,2\delta;\delta)$ are unknown in general. The
cases $\delta\mid v$ (in which
$\smax_q(v,2\delta;\delta)=\frac{q^v-1}{q^\delta-1}$ is realized by
any ($\delta-1$)-spread in $\PG(v-1,q)$)
and $\delta=2$ provide exceptions. The 
numbers $\smax_q(v,4;2)$ (maximum size of a partial line-spread in
$\PG(v-1,q)$ are known for all $q$ and $v$ (cf.\ Sections 2, 4 of
\cite{beutelspacher75} or Sections 1.1, 2.2 of
\cite{eisfeld-storme00}):
\begin{equation}
  \label{eq:beutel}
  \smax_q(v,4;2)=
  \begin{cases}
    q^{v-2}+ q^{v-4}+\dots+q^2+1&\text{if $v$ is even},\\
    q^{v-2}+ q^{v-4}+\dots+q^3+1&\text{if $v$ is odd}.
  \end{cases}
\end{equation}
In particular $\smax_q(5,4;2)=q^3+1$. Substituting this into
Lemma~\ref{lma:bound} gives the best known general
upper bound for the size of $(6,M,4;3)_q$ subspace codes, the case on
which we will focus subsequently:
\begin{lemma}
  \label{lma:bound6}
  $\smax_q(6,4;3)\leq(q^3+1)^2$.
\end{lemma}
For a subspace code $\mathcal{C}$ of size close to this upper bound
there must exist many points $P$ in $\PG(5,q)$ such that
the derived code $\mathcal{C}_P$ forms a partial spread of maximum size
$q^3+1$ in $\PG(5,q)/P\cong\PG(4,q)$. Available information 
on such partial spreads may then be used in the search for
(and classification of) optimal $(6,M,4;3)_q$ subspace codes.

The hyperplane section $\{E\cap H;E\in\mathcal{T}\}$
of a plane spread $\mathcal{T}$ in $\PG(5,q)$ with respect to any
hyperplane $H$ yields an example of a
partial spread of maximum size $q^3+1$ in $\PG(4,q)$, provided one
replaces the unique plane $E\in\mathcal{T}$ contained in the
hyperplane by a line $L\subset E$.  For the smallest case $q=2$, a
complete classification of partial spreads of size $2^3+1=9$ in
$\PG(4,2)$ is known; see \cite{shaw00,gordon-shaw-soicher04}. We will
describe this result in detail, since it forms the basis for the
computational work in Section~\ref{sec:comp}.

First let us recall that a set $\mathcal{R}$ of $q+1$ pairwise skew
lines in $\PG(3,q)$ is called a \emph{regulus} if every line in
$\PG(3,q)$ meeting three lines of $\mathcal{R}$ forms a transversal of
$\mathcal{R}$. Any three pairwise skew lines $L_1$, $L_2$, $L_3$ of
$\PG(3,q)$ are contained in a unique regulus
$\mathcal{R}=\mathcal{R}(L_1,L_2,L_3)$. The transversals of the lines
in $\mathcal{R}$ form another regulus $\mathcal{R}^\opp$, the
so-called \emph{opposite regulus}, which covers the same set of
$(q+1)^2$ points as $\mathcal{R}$.  In a similar vein, we call a set
of $q+1$ pairwise skew lines in $\PG(4,q)$ a regulus if they span a
$\PG(3,q)$ and form a regulus in their span. For $q=2$ things
simplify: A regulus in $\PG(3,2)$ is just a set of three pairwise skew
lines, and a regulus in $\PG(4,2)$ is a set of three pairwise skew
lines spanning a solid ($3$-flat).

Now suppose that $\mathcal{S}$ forms a partial line spread of size $9$
in $\PG(4,2)$. Let $Q$ be the set of $4$ \emph{holes} (points not
covered by the lines in $\mathcal{S}$).  An easy counting argument
gives that each solid $H$, being a hyperplane of
$\PG(4,2)$, contains $\alpha\in\{1,2,3\}$ lines of $\mathcal{S}$ and
$\beta\in\{0,2,4\}$ holes, where $2\alpha+\beta=6$.  This implies that
$Q=E\setminus L$ for some plane $E$ and some line $L\subset E$.  Since
a solid $H$ contains no hole iff $H\cap E=L$, there are $4$ such
solids ($\beta=0$, $\alpha=3$) and hence $4$ reguli contained in
$\mathcal{S}$. Moreover, a line $M$ of $\mathcal{S}$ is contained in
$4$ reguli if $M=L$, two reguli if $M\cap L$ is a point, or one
regulus if $M\cap L=\emptyset$. This information suffices to determine
the ``regulus pattern'' of $\mathcal{S}$, which must be either (i)
four reguli sharing the line $L$ and being otherwise disjoint
(referred to as Type~X in \cite{shaw00,gordon-shaw-soicher04}), or
(ii) one regulus containing $3$ lines through the points of $L$ and
three reguli containing one such line (Type~E), or (iii) three reguli
containing $2$ lines through the points of $L$ and one regulus
consisting entirely of lines disjoint from $E$
(Type~I$\Delta$).\footnote{Note that the letters used for the types
  look like the corresponding regulus patterns.}

The partial spread obtained from a plane spread in $\PG(5,2)$ as
described above has Type~X. Replacing one of its reguli by the
opposite regulus gives a partial spread of Type~E. A partial spread of
Type~I$\Delta$ can be constructed as follows: Start with $3$ lines
$L_1$, $L_2$, $L_3$ in $\PG(4,2)$ not forming a regulus (i.e.\ not
contained in a solid). The lines $L_i$ have a unique transversal line
$L$ (the intersection of the three solids $H_1=\langle
L_2,L_3\rangle$, $H_2=\langle L_1,L_3\rangle$, $H_3=\langle
L_1,L_2\rangle$).\footnote{This property remains true for any
  $q$. Thus three pairwise skew lines in $\PG(4,q)$ have $q+1$
  transversals if they are contained in a solid, and a unique
  transversal otherwise.}  In $\PG(4,2)/L\cong\PG(2,2)$ the solids
$H_1$, $H_2$, $H_3$ form the sides of a triangle. Hence there exist a
unique plane $E$ and a unique solid $H_4$ in $\PG(4,2)$ such that
$H_i\cap E=L$ for $1\leq i\leq 4$. Let $E_i=\langle L,L_i\rangle$ for
$1\leq i\leq 3$ (the planes in which the solids $H_1$, $H_2$, $H_3$
intersect). Choose $3$ further lines $L_i'\subset H_i$, $1\leq i\leq
3$, such that $L_1$, $L_2$, $L_3$, $L_1'$, $L_2'$, $L_3'$ are pairwise
skew.\footnote{This can be done in $8$ possible ways, ``wiring'' the
  $6$ points in $(E_1\cup E_2\cup E_3)\setminus(L_1\cup L_2\cup
  L_3)$. All these choices lead to isomorphic partial spreads of size
  $6$, as can be seen using the fact that $(L_1\cup L_2\cup
  L_3)\setminus L$ forms a projective basis of $\PG(4,2)$.} Let
$y_i=L_i'\cap H_4$ for $1\leq i\leq 3$. Then
$L'=\{y_1,y_2,y_3\}\subset H_4$ must be a line, since it is the only
candidate for the transversal to $L_1'$, $L_2'$, $L_3'$ (which span
the whole geometry).\footnote{The image of the transversal in
  $\PG(4,2)/L$ must be a line, leaving only one choice for the
  transversal.} Hence $H_4\cong\PG(3,2)$ contains the disjoint lines
$L$, $L'$ and $9$ points not covered by the $6$ lines chosen so far. A
partial spread of Type I$\Delta$ can now be obtained by completing
$L$, $L'$ to a spread of $H_4$.

With a little more effort, one can show that there are exactly $4$
isomorphism types of partial spreads of size $9$ in $\PG(4,2)$, one of
Type~X, one of Type~E, and two of Type~I$\Delta$, represented by the
two possible ways to complete $L$, $L'$ to spread of $H_4$; see
\cite{shaw00,gordon-shaw-soicher04} for details.

\subsection{A property of the reduced row echelon form}
Subspaces $U$ of $\F_q^v$ with $\dim(U)=k$ are conveniently
represented by matrices $\mat{U}\in\F_q^{k\times v}$ with row space
$U$; notation $U=\langle\mat{U}\rangle$. It is well-known that every
$k$-dimensional subspace of $\F_q^v$ has a unique representative
$\mat{U}$ in \emph{reduced row echelon form}, which is defined as
follows: $\mat{U}=(u_{ij})=(\vek{u}_1|\vek{u}_2|\dots|\vek{u}_v)$
contains the $k\times k$ identity matrix $\imat_k$ as a submatrix,
i.e.\ there exist column indices $1\leq j_1<j_2<\dots<j_k\leq v$
(``pivot columns'') such that $\vek{u}_{j_i}=\uvek_i$ (the $i$-th
standard unit vector in $\F_q^k$) for $1\leq i\leq k$, and every entry
to the left of $u_{i,j_i}=1$ (``pivot element'') is
zero.\footnote{Although $\imat_k$ may occur multiple times as a
  submatrix of $\mat{U}$, both properties together determine the set
  $\{j_1,\dots,j_k\}$ of pivot columns uniquely.}  
The reduced row echelon form of $\{\vek{0}\}$ is defined as the ``empty
  matrix'' $\emptyset$.

  We will call $\mat{U}$ the \emph{canonical matrix} of $U$ and write
  $\mat{U}=\cm(U)$. Denoting the set of all canonical matrices
  (matrices in reduced row echelon form without all-zero rows) over
  $\F_q$ with $v$ columns (including the empty matrix) by
  $\cms=\cms(v,q)$ and writing
  $\mathcal{L}=\subspaces(v,q)=\subspaces(\F_q^v)$, we have that
  $\mathcal{L}\to\cms$, $U\mapsto\cm(U)$ and $\cms\to\mathcal{L}$,
  $\mat{U}\mapsto\langle\mat{U}\rangle$ are mutually inverse
  bijections.

  In addition we consider $\cms(\infty,q)=\bigcup_{v=1}^\infty\cms(v,q)$, the
  set of all matrices over $\F_q$ in canonical form.
  The following property of $\cms(\infty,q)$ seems not to
  be well-known, but forms the basis for determining the
  canonical matrices of all subspaces of $U$ from $\cm(U)$; see
  Part~(ii).
  \begin{lemma}
    \label{lma:cm}
    \begin{enumerate}[(i)]
    \item\label{lma:cm:mult_closed} $\cms(\infty,q)$ is closed with respect to multiplication
      whenever it is defined; i.e., if
      $\mat{Z},\mat{U}\in\cms(\infty,q)$ and the number of columns of
      $\mat{Z}$ equals the number of rows of $\mat{U}$, then
      $\mat{ZU}\in\cms(\infty,q)$ as well.
    \item\label{lma:cm:subspaces} Let $U$ and $V$ be subspaces of $\F_q^v$ with $V\subseteq
      U$. Then $\cm(V)=\mat{Z}\cdot\cm(U)$ for a unique matrix
      $\mat{Z}$ in reduced row echelon form. More precisely, if
      $\dim(U)=k$ then $\mat{Z}\mapsto\mat{Z}\cdot\cm(U)$ defines a
      bijection from $\cms(k,q)$ to the set of all matrices in
      $\cms(v,q)$ that represent subspaces of $U$, and
      $\langle\mat{Z}\rangle\mapsto\langle\mat{Z}\cdot\cm(U)\rangle$
      defines a bijection from the set of subspaces of $\F_q^k$ to the
      set of subspaces of $U$.
    \end{enumerate}
  \end{lemma}

  \begin{proof}
    (ii) follows easily from (i). For the proof of (i) assume
    $\mat{Z}=(z_{ri})\in\F_q^{l\times k}$, $\mat{U}=(u_{ij})\in\F_q^{k\times v}$, and
    $\mat{U}$ has pivot columns $j_1<\dots<j_k$. Then $\mat{Z}$
    appears as a submatrix of $\mat{ZU}$ in columns $j_1,\dots,j_k$,
    and hence $\imat_l$ appears as a submatrix of $\mat{ZU}$ as well,
    viz.\ in columns $j_{i(1)},j_{i(2)},\dots,j_{i(l)}$, where
    $i(1)<i(2)<\dots<i(l)$ are the pivot columns of $\mat{Z}$. Now
    suppose $j<j_{i(r)}$, so that the entry
    $(\mat{ZU})_{rj}=\sum_{i=1}^kz_{ri}u_{ij}$ lies to the left of
    the $r$-th pivot element of $\mat{ZU}$. Since $z_{ri}=0$ for
    $i<i(r)$, we have
    $(\mat{ZU})_{rj}=\sum_{i=i(r)}^kz_{ri}u_{ij}$. But $i\geq i(r)$
    implies $j_i\geq j_{i(r)}>j$ and hence $u_{ij}=0$. Thus
    $(\mat{ZU})_{rj}=0$, proving that $\mat{ZU}$
    is indeed canonical. 
  \end{proof}

  It is worth noting that for a given subspace $U$ of $\F_q^v$
  represented by its canonical matrix,
  Lemma~\ref{lma:cm}\ref{lma:cm:subspaces} provides an efficient
  way to enumerate all subspaces of $U$ (or all subspaces of $U$ of a
  fixed dimension).

\subsection{Geometry of Rectangular Matrices}\label{sec:lag}
Left multiplication with matrices in $\Mat(m,q)$ (the ring of $m\times
m$ matrices over $\F_q$) endows the $\F_q$-space $\F_q^{m\times n}$ of $m\times n$
matrices over $\F_q$ with the structure of a (left)
$\Mat(m,q)$-module. Following \cite{yt:JSCC7}, we call the
corresponding coset geometry \emph{left affine geometry of $m\times n$
matrices over $\F_q$} and denote it by $\LAG(m,n,q)$.\footnote{This
viewpoint is different from that in the classical Geometry of
Matrices \cite[Ch.~3]{wan96a}, which studies the structure of
$\F_q^{m\times n}$ as an $\Mat(m,q)$-$\Mat(n,q)$ bimodule.} For $0\leq
k\leq n$, the $k$-flats of $\LAG(m,n,q)$ are the cosets
$\mat{A}+\mathcal{U}$, where $\mat{A}\in\F_q^{m\times n}$ and the
submodule $\mathcal{U}$ consists of all matrices
$\mat{U}\in\F_q^{m\times n}$ whose row space $\langle\mat{U}\rangle$
is contained in a fixed $k$-dimensional subspace of $\F_q^n$. The size
of a $k$-flat is $\#(\mat{A}+\mathcal{U})=\#\mathcal{U}=q^{mk}$,
and the number of $k$-flats in $\LAG(m,n,q)$ is
$q^{m(n-k)}\gauss{n}{k}{q}$.\footnote{The $k$-flats fall into $\gauss{n}{k}{q}$
parallel classes, each consisting of $q^{m(n-k)}=q^{mn}/q^{mk}$ flats.}
Point residues $\LAG(m,n,q)/\mat{A}$ are isomorphic to $\PG(n-1,q)$.

Our interest in these geometries comes from the following
\begin{lemma}
  \label{lma:lag}
  $\mat{A}\mapsto\langle(\imat_m|\mat{A})\rangle$ maps $\LAG(m,n,q)$
  isomorphically onto the subgeometry $\attspace$ of $\PG(m+n-1,q)$
  whose $k$-flats, $0\leq k\leq n$, are the $(m+k)$-dimensional
  subspaces of $\F_q^{m+n}$ meeting
  $S=\langle\uvek_{m+1},\dots,\uvek_{m+n}\rangle$ in a subspace of
  dimension $k$. A parallel class
  $\{\mat{A}+\mathcal{U};\mat{A}\in\F_q^{m\times n}\}$ of $k$-flats is
  mapped to the set of ($m+k$)-dimensional subspaces of $\F_q^{m+n}$
  meeting $S$ in $U=\sum_{\mat{U}\in\mathcal{U}}\langle\mat{U}\rangle$
  (with coordinates rearranged in the obvious way), so that $U$
  represents the common $k$-dimensional ``space at infinity'' of the
  flats in $\mat{A}+\mathcal{U}$.\footnote{Note, however, that in
    contrast with the classical case $m=1$ of projective geometry over
    fields, points of $\attspace$ need not be either ``finite'' or
    ``infinite'': The pivots of the associated $m$-dimensional
    subspace of $\F_q^{m+n}$ may involve both parts of the coordinate
    partition $\{1,\dots,m\}\uplus\{m+1,\dots,m+n\}$.}
\end{lemma}
In particular the points of $\LAG(m,n,q)$ correspond to the
($m-1$)-flats of $\PG(m+n-1,q)$ disjoint from the special ($n-1$)-flat
$S$, and the lines of $\LAG(m,n,q)$ correspond to the $m$-flats of
$\PG(m+n-1,q)$ meeting $S$ in a point. The lemma seems to be
well-known (cf.\ \cite[Ex.~1.5]{bonoli-melone03},
\cite[2.2.7]{clerck-maldeghem95}), but we include a proof for
completeness.
\begin{proof}
  For an $(m+k)$-dimensional subspace $F$ of $\F_q^{m+n}$ the condition
  $\dim(F\cap S)=k$ is equivalent to $\cm(F)=\left(
    \begin{smallmatrix}
      \imat_m&\mat{A}\\
      \mat{0}&\mat{Z}
    \end{smallmatrix}
\right)$ for some $\mat{A}\in\F_q^{m\times n}$ and some canonical
$\mat{Z}\in\F_q^{k\times n}$. The $m$-dimensional subspaces
$E\subseteq F$ with $E\cap S=\{\vek{0}\}$ are precisely those with
$\cm(E)=(\imat_m|\mat{A}+\mat{VZ})$ for some $\mat{V}\in\F_q^{m\times
  k}$. Clearly $\{\mat{VZ};\mat{V}\in\F_q^{m\times
  k}\}=\{\mat{U}\in\F_q^{m\times
  n};\langle\mat{U}\rangle\subseteq\langle\mat{Z}\rangle\}=\mathcal{U}$,
say, is a $k$-flat in $\LAG(m,n,q)$ (containing
$\mat{0}\in\F_q^{m\times n}$). This
shows that $F$ contains precisely those $E$ which correspond
to points on the $k$-flat $\mat{A}+\mathcal{U}$. The induced map on
$k$-flats is clearly bijective, and maps the parallel class of
$\mat{A}+\mathcal{U}$ to the subspaces $F$ intersecting $S$ in
$\langle\mat{Z}\rangle=\sum_{\mat{U}\in\mathcal{U}}\langle\mat{U}\rangle$. The
result follows.\footnote{Since
  $\mat{A}$ is required to have $k$ zero columns (``above the pivots of
  $\mat{Z}$''), the number of different choices for $\mat{A}$ is
  indeed $q^{m(n-k)}$.}
\end{proof}

\subsection{Maximum Rank Distance Codes}\label{ssec:mrd}
The set $\F_q^{m\times n}$ of $m\times n$ matrices over $\F_q$ forms a
metric space with respect to the \emph{rank distance} defined by
$\rdist(\mat{A},\mat{B})=\rank(\mat{A}-\mat{B})$. As shown in
\cite{delsarte78a,gabidulin85,roth91}, the maximum size of a code of
minimum distance $d$, $1\leq d\leq\min\{m,n\}$, in $(\F_q^{m\times
  n},\rdist)$ is $q^{n(m-d+1)}$ for $m\leq n$ and $q^{m(n-d+1)}$ for
$m\geq n$. A code $\mathcal{A}\subseteq\F_q^{m\times n}$ meeting this
bound with equality is said to be a $q$-ary $(m,n,k)$ \emph{maximum
  rank distance (MRD) code}, where $k=m-d+1$ for $m\leq n$ and
$k=n-d+1$ for $m\geq n$ (i.e.\ $\#\mathcal{A}=q^{nk}$ resp.\
$\#\mathcal{A}=q^{mk}$).\footnote{MRD codes form in some sense the
  $q$-analogue of maximum distance separable (MDS) codes. For more on
  this analogy see \cite{delsarte78a,yt:JSCC7} and Lemma~\ref{lma:mrd}
  below.}

>From now on we will always assume
that $m\leq n$.\footnote{The case $m\geq n$ readily reduces to $m\leq
  n$ by transposing $\mat{A}\to\mat{A}^\tp$, which maps
  $(\F_q^{m\times n},\rdist)$ isometrically onto $(\F_q^{n\times
    m},\rdist)$.}  Examples of MRD codes are the \emph{Gabidulin
  codes}\footnote{Although this name is commonly used now, it should be
  noted that the examples found in \cite{delsarte78a,roth91} are
  essentially the same.}, which can be defined as follows: Consider
the $\F_q$-space $V=\End(\F_{q^n}/\F_q)$ of all $\F_q$-linear
endomorphisms of the extension field $\F_{q^n}$. Then $V$ is also a
vector space over $\F_{q^n}$ (of dimension $n$), and elements of $V$
are uniquely represented as linearized polynomials (``polynomial
functions'') $x\mapsto
a_0x+a_1x^q+a_2x^{q^2}+\dots+a_{n-1}x^{q^{n-1}}$ with coefficients
$a_i\in\F_{q^n}$ and $q$-degree $<n$. The $(n,n,k)$ Gabidulin code
$\mathcal{G}$ consists of all such polynomials of $q$-degree $<k$. The
usual matrix representation of $\mathcal{G}$ is obtained by choosing
coordinates with respect to a fixed basis of $\F_{q^n}/\F_q$. This
gives rise to an isomorphism $(V,\rdist)\cong(\F_q^{n\times
  n},\rdist)$ of metric spaces, showing that the choice of basis does
not matter and the coordinate-free representation introduced above is
equivalent to the matrix representation. Rectangular $(m,n,k)$
Gabidulin codes (where $m<n$) are then obtained by restricting the
linear maps in $\mathcal{G}$ to an $m$-dimensional $\F_q$-subspace
$W$ of $\F_{q^n}$.

\begin{example}
  \label{ex:332q}
  Our main interest is in the case $m=n=3$, $d=2$. Here
  $\mathcal{G}=\{a_0x+a_1x^q;a_0,a_1\in\F_{q^3}\}$. The $2(q^3-1)$
  monomials $ax$ and $ax^q$, $a\in\F_{q^3}^\times$, have rank
  $3$. The remaining nonzero elements of $\mathcal{G}$ are of the form
  $a(x^q-bx)$ with $a,b\in\F_{q^3}^\times$ and have
  \begin{equation*}
    \rank\bigl(a(x^q-bx)\bigr)=
    \begin{cases}
      2&\text{if $b^{q^2+q+1}=1$},\\
      3&\text{if $b^{q^2+q+1}\neq 1$}.
    \end{cases}
  \end{equation*}
  This is easily seen by looking at the corresponding kernels:
  $x^q-bx=0$ has a nonzero solution iff $b=u^{q-1}$ for some
  $u\in\F_{q^3}^\times$, in which case $\kernel(x^q-bx)=\F_qu$. The
  rank distribution of $\mathcal{G}$ is shown in Table~\ref{tbl:rankdist}.
  \begin{table}[ht]
    \centering
    \(\begin{array}{c||c|c|c|c}
      \text{rank}&0&1&2&3\\\hline
      \#\text{codewords}&1&0&(q^3-1)(q^2+q+1)&(q^3-1)(q^3-q^2-q)
    \end{array}\)
    \caption{Rank distribution of the $q$-ary $(3,3,2)$ Gabidulin code}
    \label{tbl:rankdist}
  \end{table}
\end{example}

Now we return to the case of arbitrary MRD codes and state one of
their fundamental properties.\footnote{This property is analogous to
  the fact that for an $[n,k]$ MDS code every set of $k$ coordinates
  forms an information set. One might say that every $k$-dimensional
  subspace of $\F_q^m$ forms an ``information subspace'' for
  $\mathcal{A}$.}

\begin{lemma}[{cf.\ \cite[Lemma~2.4]{yt:JSCC7}}]
  \label{lma:mrd}
  Let $\mathcal{A}\subseteq\F_q^{m\times n}$ be an $(m,n,k)$ MRD code,
  where $m\leq n$, and $\mat{Z}\in\F_q^{k\times m}$ of full rank
  $k$. Then $\mathcal{A}\to\F_q^{k\times n}$, $\mat{A}\mapsto\mat{ZA}$
  is a bijection.\footnote{In particular, if $\mathcal{A}$ is linear (i.e.\ an
    $\F_q$-subspace of $\F_q^{m\times n}$) then
    $\mat{A}\mapsto\mat{ZA}$ is an isomorphism of
    $\F_q$-vector spaces.}
\end{lemma}
The matrix-free version of this easily proved lemma says that
restriction to an arbitrary $k$-dimensional $\F_q$-subspace
$U\subseteq W\subseteq\F_{q^n}$ maps an $(m,n,k)$ MRD code
$\mathcal{A}\subseteq\Hom_{\F_q}(W,\F_{q^n})$ isomorphically onto
$\Hom_{\F_q}(U,\F_{q^n})$.

\subsection{Lifted Maximum Rank Distance Codes}\label{ssec:lmrd}
By a \emph{lifted maximum rank distance (LMRD) code} we mean a
subspace code obtained from an MRD code
$\mathcal{A}\subseteq\F_q^{m\times n}$ by the so-called \emph{lifting
  construction} of \cite{silva-kschischang-koetter08}, which assigns
to every matrix $\mat{A}\in\F_q^{m\times n}$ the subspace
$U=\langle(\imat_m|\mat{A})\rangle$ of $\F_q^{m\times(m+n)}$. The map
$\mat{A}\mapsto\langle(\imat_m|\mat{A})\rangle$ defines an isometry with
scale factor $2$ from $\bigl(\F_q^{m\times n},\rdist\bigr)$ into
$\bigl(\gaussm{V}{m},\sdist\bigr)$, $V=\F_q^{m+n}$; cf.\
\cite[Prop.~4]{silva-kschischang-koetter08}.
  
Rewriting everything in terms
of the subspace code parameters $v$, $d$, and $k$, a $q$-ary
$(v,M,d;k)$ constant-dimension subspace code $\mathcal{L}$ is an LMRD
code if $\mathcal{L}=\bigl\{\langle(\imat_k|\mat{A})\rangle;
\mat{A}\in\mathcal{A}\bigr\}$ for some MRD code
$\mathcal{A}\subseteq\F_q^{k\times(v-k)}$ of minimum distance
$d/2$. If $v\geq 2k$ (the case of interest to us) then the MRD code
parameters of $\mathcal{A}$ are $(m,n,k')
=(k,v-k,k-d/2+1)$, and
$M=\#\mathcal{L}=\#\mathcal{A}=q^{(v-k)(k-d/2+1)}$.

As for ordinary MRD codes, it is sometimes more convenient to use a
coordinate-free representation of LMRD codes, obtained as follows:
Suppose $V$, $W$ are vector spaces over $\F_q$ with $\dim(V)=n$,
$\dim(W)=m$ and $\mathcal{A}\subseteq\Hom(W,V)$ is an $(m,n,k)$ MRD
code. Let $\mathcal{L}$ be the subspace code with ambient space
$W\times V$ and members $G(f)=\left\{(x,f(x));x\in W\right\}$ for all
  $f\in\mathcal{A}$ (``graphs'' of the functions in
  $\mathcal{A}$). Since each $f\in\mathcal{A}$ is $\F_q$-linear, it is
  obvious that $\mathcal{L}$ consists of $m$-dimensional subspaces of
  $W\times V$. Moreover, choosing bases of $V$ and $W$ and representing
  linear maps $f\colon W\to V$ by $m\times n$ matrices $\mat{A}$ (acting on row
  vectors) with respect to these bases, induces an $\F_q$-isomorphism
  $V\times W\to\F_q^{m+n}$ sending $G(f)$ to
  $\langle(\imat_m|\mat{A})\rangle$. The induced isometry
  $\subspaces(V\times W)\to\subspaces(\F_q^{m+n})$ maps
  $\mathcal{L}$ to an LMRD code with ambient space $\F_q^{m+n}$,
  showing that both views are equivalent.

\begin{example}
  \label{ex:332qcont}
  The $q$-ary $(3,3,2)$ Gabidulin code $\mathcal{G}$ of
  Example~\ref{ex:332q} lifts to a $(6,q^6,4;3)_q$ constant-dimension
  subspace code $\mathcal{L}$. In the coordinate-free representation,
  the members of $\mathcal{L}$ are the subspaces
  $G(a_0,a_1)=\left\{(x,a_0x+a_1x^q);x\in\F_{q^3}\right\}
  \in\subspaces(\F_{q^3}\times\F_{q^3})$, where $a_0,a_1\in\F_{q^3}$.
    The ambient space $\F_{q^3}\times\F_{q^3}$ is considered as a
    vector space over $\F_q$.

  A (coordinate-dependent) representation of
  $\mathcal{L}$ by $3\times 6$ matrices over $\F_q$ is obtained by
  choosing a basis $(b_1,b_2,b_3)$ of $\F_{q^3}/\F_q$ and writing
  $a_0x+a_1x^q\in\mathcal{G}$ with respect to this basis. For example,
  if $q=2$ then we can choose the basis $(\beta,\beta^2,\beta^4)$, where
  $\beta^3+\beta^2+1=0$ (the unique normal basis of $\F_8/\F_2$). Then
  \begin{equation*}
    \beta\leftrightarrow\mat{B}=
    \begin{pmatrix}
      0&1&0\\
      1&0&1\\
      0&1&1
    \end{pmatrix},\qquad
    x^2\leftrightarrow\mat{\Phi}=
    \begin{pmatrix}
      0&1&0\\
      0&0&1\\
      1&0&0
    \end{pmatrix},
  \end{equation*}
  and $\mathcal{L}$ consists of the $64$ subspaces
  $\langle(\imat_3|\mat{A})\rangle$ of $\F_2^6$ corresponding to
  $\mat{A}=\mat{0}$, $\mat{B}^i$, $\mat{B}^i\mat{\Phi}$,
  $\mat{B}^i+\mat{B}^j\mat{\Phi}$ with $0\leq i,j\leq 6$.
\end{example}
Viewed as point sets in $\PG(v-1,q)$, the members of a $q$-ary
$(v,M,d;k)$ LMRD code are disjoint from the special ($v-k-1$)-flat
$S=\langle\uvek_{k+1},\uvek_{k+2},\dots,\uvek_v\rangle$. Delving
further into this, one finds that under the hypothesis $v\geq 2k$
LMRD codes with $d=2k$ partition the
points of $\PG(v-1,q)$ outside $S$, and LMRD codes with $d<2k$ form
higher-dimensional analogs of such partitions:

\begin{lemma}
  \label{lma:lmrd}
  Let $\mathcal{L}$ be a $q$-ary $(v,M,d;k)$ LMRD code with $v\geq 2k$
  and set $t=k-d/2+1$ (so that $\mathcal{L}$ arises from a $q$-ary
  $(k,v-k,t)$ MRD code $\mathcal{A}$ by the lifting
  construction). Then the members of $\mathcal{L}$ cover every
  ($t-1$)-flat in $\PG(v-1,q)$ disjoint from the special flat $S$
  exactly once.
\end{lemma}
\begin{proof}
  If a ($t-1$)-flat in $\PG(v-1,q)$ were covered twice, we would
  have $\sdist(\mathcal{L})\leq 2k-2t=d-2$, a
  contradiction. Now suppose $F$ is a ($t-1$)-flat in $\PG(v-1,q)$
  disjoint from $S$. Since the pivot columns of $\cm(F)$ are within
  $\{1,\dots,k\}$, we must have $\cm(F)=(\mat{Z}|\mat{B})$ for some
  canonical matrix $\mat{Z}\in\F_q^{t\times k}$ and some matrix
  $\mat{B}\in\F_q^{t\times(v-k)}$. By Lemma~\ref{lma:mrd},
  $\mat{B}=\mat{ZA}$ for some $\mat{A}\in\mathcal{A}$. Hence
  $\cm(F)=(\mat{Z}|\mat{ZA})=\mat{Z}(\imat_k|\mat{A})$, so that $F$
  is covered by
  the codeword $\langle(\imat_k|\mat{A})\rangle\in\mathcal{L}$.
\end{proof}

Returning to Example~\ref{ex:332qcont}, the lemma says that the $q^6$
planes in $\PG(5,q)$ obtained by lifting the $q$-ary $(3,3,2)$
Gabidulin code cover each of the $q^6(q^2+q+1)$ lines in $\PG(5,q)$
disjoint from the special plane
$S=\langle\uvek_4,\uvek_5,\uvek_6\rangle$ exactly once (and thus in
particular the number of lines in $\PG(5,q)$ disjoint from $S$ equals
$q^6(q^2+q+1)$). As an immediate consequence of this we have that the
$(6,q^6,4;3)_q$ LMRD code $\mathcal{L}$ formed by these planes cannot
be enlarged, without decreasing the minimum distance, by adding a new
plane $E$ with $\dim(E\cap S)\leq 1$. Indeed, such a plane would
contain a line disjoint from $S$, and hence already covered by
$\mathcal{L}$, producing a codeword $U\in\mathcal{L}$ with
$\sdist(E,U)=2$. Since lines contained in $S$ are covered by at most
one codeword of $\mathcal{C}$, we obtain for any $(6,M,4;3)_q$
subspace code $\mathcal{C}$ containing $\mathcal{L}$ the upper bound
$M\leq q^6+q^2+q+1$ (a special case of
\cite[Th.~11]{etzion-silberstein13}). This shows already that optimal
$(6,M,4;3)_2$ subspace codes, which have $M>71$, cannot contain an
LMRD subcode.
  
\section{The Computer Classification}\label{sec:comp}
For the classification of optimal $(6,M,4;3)_2$ subspace codes,
we are facing a huge search space, making a direct attack by a
depth-first search or an integer linear program infeasible.  We will
bring this under control by using certain substructures of large
$(6,M,4;3)_2$ subspace codes as intermediate classification steps.
Furthermore, we make use of the group $\GL(6,2)$ of order
$20158709760$, which is acting on the search space.  The canonization
algorithm in~\cite{feulner13} (based on~\cite{feulner09}, see also~\cite{feulner-diss}) will be used
to reject isomorphic copies at critical points during the search,
keeping us from generating the ``same'' object too many times.  The
same method allows us to compute automorphism groups of network codes,
and to filter the eventual list of results for isomorphic copies.

Given a $(6,M,4;3)_2$
constant-dimension code $\mathcal{C}$ with ambient space $V = \F_2^6$, 
we define the \emph{degree} of a
point $P\in\gaussm{V}{1}$ as 
\[
    r(P) = \#\{E\in\mathcal{C} \mid P \subset E\}\text{.}
\]
By the discussion in Section~\ref{ssec:bound}, $r(P) \leq 9$ for all $P$.

\subsection{Classification of $9$-configurations}
A subset of $\mathcal{C}$ consisting of $9$ planes passing through a
common point $P$ will be called a \emph{$9$-configuration}.
Obviously, points of degree $9$ in $\PG(V)$ correspond to
$9$-configurations in $\mathcal{C}$.

\begin{lemma}
\label{lma:9conf}
If $\#\mathcal{C} \geq 73$ then $\mathcal{C}$ contains a $9$-configuration.
\end{lemma}

\begin{proof}
  Assuming $r(P) \leq 8$ for all points $P$ and double counting the
  pairs $(P,E)$ with points $P\in\gaussm{V}{1}$ and codewords
  $E\in\mathcal{C}$ passing through $P$ yields $7\cdot\#\mathcal{C}
  \leq 8\cdot 63$ and hence the contradiction $\#\mathcal{C} \leq 72$.
\end{proof}

Lemma~\ref{lma:9conf} implies that for the classification of optimal
$(6,M,4;3)_2$ codes $\mathcal{C}$, we may assume the
existence of a $9$-configuration in $\mathcal{C}$.
By Section~\ref{ssec:bound}, the derived code $\mathcal{C}_P$ in any
point $P$ of degree $9$ is one of the four isomorphism types of
partial spreads in $\PG(V/P)\cong\PG(4,2)$
and will accordingly be denoted by X, E, I$\Delta$ or I$\Delta'$.

The above discussion shows that for our classification goal, we may
start with the four different $9$-configurations and enumerate all
extensions to a $(6,M,4;3)_2$ code with $M\geq 77$.
However, the search space is still too large to make this approach feasible.
Consequently, another intermediate classification goal is needed.

\subsection{Classification of $17$-configurations}
A subset of $\mathcal{C}$ of size $17$ will be called a
\emph{$17$-configuration} if it is the union of two
$9$-con\-fi\-gu\-ra\-tions.  A $17$-configuration corresponds to a pair of
points $(P,P')$ of degree $9$ that is connected by a codeword in
$\mathcal{C}$.  The next lemma shows that large $(6,M,4;3)_2$ codes
are necessarily extensions of a $17$-configuration.

\begin{lemma}
\label{lma:17conf}
If $\#\mathcal{C} \geq 74$ then $\mathcal{C}$ contains a $17$-configuration.
\end{lemma}

\begin{proof}
  By Lemma~\ref{lma:9conf} there is a point $P$ of degree $9$.  
  The $9$-configuration through $P$ covers $9\cdot 6 = 54$ points $P'\neq
  P$, since two of its planes cannot have more than a single point in
  common. Under the assumption that there is no $17$-configuration, we
  get that those $54$ points are of degree $\leq 8$.  Double counting
  the set of pairs $(Q,E)$ of points $Q$ and codewords $E$ passing
  through $Q$ shows that
\[
7\cdot\#\mathcal{C} \leq 54\cdot 8 + (63 - 54)\cdot 9
\]
and hence $\#\mathcal{C} \leq \frac{513}{7} < 74$, a
contradiction.
\end{proof}

For the classification of $17$-configurations, we start with the four
isomorphism types of $9$-configurations.  In the following,
$\mathcal{N}$ denotes a $9$-configuration, $P=\bigcap\mathcal{N}$ the
intersection point of its $9$ planes and
$M=\bigcup\mathcal{N}\setminus \{P\}$ the set of $54$ points distinct from
$P$ covered by a block in $\mathcal{N}$.
Up to isomorphism, the possible choices for the second intersection point $P'$ are
given by the orbits of the automorphism group of $\mathcal{N}$ on $M$.
This orbit structure is shown in Table~\ref{tbl:orbN}.  For example,
for type E the $48$ points in $M$ fall into $4$ orbits of length $12$
and a single orbit or length $6$, so up to isomorphism, there are $5$
ways to select the point $P'$.  For each Type
$T\in\{\text{X},\text{E},\text{I}\Delta,\text{I}\Delta'\}$, the
different types for the choice of $P'$ will be denoted by $T_1,
T_2,\ldots$, enumerating those coming from larger orbits first, see
Table~\ref{tbl:ext9conf}.

\begin{table}[htbp]
  \centering
  \(
  \begin{array}{ccc}
  	\text{$\mathcal{N}$} & \#\Aut(\mathcal{N}) & \text{orbit structure on $M$} \\
	\hline
	\text{X}        & 48 & 48^1 6^1 \\
	\text{E}        & 12 & 12^4 6^1 \\
	\text{I}\Delta  & 12 & 12^2 6^5 \\
	\text{I}\Delta' & 12 & 12^2 6^5
  \end{array}
  \)
  \caption{Orbit structure on the points covered by a $9$-configuration $\mathcal{N}$} 
  \label{tbl:orbN}
\end{table}

These $2 + 5 + 7 + 7 = 21$ different \emph{dotted $9$-structures} 
($9$-structures together with the selected point $P'$) are
fed into a depth-first search enumerating all extensions to a
$17$-configuration by adding blocks through $P'$.  For each case, the
number of extensions is shown in Table~\ref{tbl:ext9conf}.  Filtering
out isomorphic copies among the resulting $575264$
$17$-configurations, we arrive at

\begin{lemma}
There are $12770$ isomorphism types of $17$-configurations.
\end{lemma}

\begin{table}[htbp]
  \centering
  \(
  \begin{array}{crr}
  	\text{dotted $9$-conf.} & \text{orbit size} & \#\text{extensions} \\
	\hline
\text{X}_1        & 48 & 28544 \\
\text{X}_2        &  6 & 23968 \\
\text{E}_1        & 12 & 28000 \\
\text{E}_2        & 12 & 28544 \\
\text{E}_3        & 12 & 28000 \\
\text{E}_4        & 12 & 27168 \\
\text{E}_5        &  6 & 25632 \\
\text{I}\Delta_1  & 12 & 28544 \\
\text{I}\Delta_2  & 12 & 28000 \\
\text{I}\Delta_3  &  6 & 27168 \\
\text{I}\Delta_4  &  6 & 27680 \\
\text{I}\Delta_5  &  6 & 27680 \\
\text{I}\Delta_6  &  6 & 25632 \\
\text{I}\Delta_7  &  6 & 28256 \\
\text{I}\Delta'_1 & 12 & 28544 \\
\text{I}\Delta'_2 & 12 & 28000 \\
\text{I}\Delta'_3 &  6 & 27168 \\
\text{I}\Delta'_4 &  6 & 27680 \\
\text{I}\Delta'_5 &  6 & 27680 \\
\text{I}\Delta'_6 &  6 & 25632 \\
\text{I}\Delta'_7 &  6 & 27744 \\
\hline
 & &                   575264
  \end{array}
  \)
  \caption{Extension of dotted $9$-configurations to $17$-configurations}
  \label{tbl:ext9conf}
\end{table}

Each $17$-configuration contains two dotted $9$-configurations.  The
exact distribution of the pairs of isomorphism types of dotted
$9$-con\-fi\-gu\-ra\-tions is shown in Table~\ref{tbl:pairs9in17}.  For
example, there are $68$ isomorphism types of $17$-configurations such
that both dotted $9$-configurations are of isomorphism type
$\text{X}_1$.

\begin{table}
	\centering
	\resizebox{\textwidth}{!}{
\(
\begin{array}{c|r*{20}{@{\hspace{3pt}}r}}
&\text{X}_1&\text{X}_2&\text{E}_1&\text{E}_2&\text{E}_3&\text{E}_4&\text{E}_5&\text{I}\Delta_1&\text{I}\Delta_2&\text{I}\Delta_3&\text{I}\Delta_4&\text{I}\Delta_5&\text{I}\Delta_6&\text{I}\Delta_7&\text{I}\Delta'_1&\text{I}\Delta'_2&\text{I}\Delta'_3&\text{I}\Delta'_4&\text{I}\Delta'_5&\text{I}\Delta'_6&\text{I}\Delta'_7\\
\hline
\text{X}_1&68&10&120&132&120&120&58&132&120&60&59&59&56&60&132&120&60&59&59&56&60\\
\text{X}_2&&3&16&10&16&12&8&10&16&7&8&8&6&10&10&16&7&8&8&6&10\\
\text{E}_1&&&64&120&124&114&54&120&124&57&60&60&56&60&120&124&57&60&60&56&64\\
\text{E}_2&&&&68&120&120&58&132&120&60&59&59&56&60&132&120&60&59&59&56&60\\
\text{E}_3&&&&&64&114&54&120&124&57&60&60&56&60&120&124&57&60&60&56&64\\
\text{E}_4&&&&&&62&54&120&114&60&58&58&58&60&120&114&60&58&58&58&48\\
\text{E}_5&&&&&&&14&58&54&27&32&32&20&32&58&54&27&32&32&20&28\\
\text{I}\Delta_1&&&&&&&&68&120&60&59&59&56&60&132&120&60&59&59&56&60\\
\text{I}\Delta_2&&&&&&&&&64&57&60&60&56&60&120&124&57&60&60&56&64\\
\text{I}\Delta_3&&&&&&&&&&17&29&29&29&31&60&57&31&29&29&29&24\\
\text{I}\Delta_4&&&&&&&&&&&17&32&31&32&59&60&29&32&32&31&28\\
\text{I}\Delta_5&&&&&&&&&&&&17&31&32&59&60&29&32&32&31&28\\
\text{I}\Delta_6&&&&&&&&&&&&&13&24&56&56&29&31&31&23&28\\
\text{I}\Delta_7&&&&&&&&&&&&&&23&60&60&31&32&32&24&23\\
\text{I}\Delta'_1&&&&&&&&&&&&&&&68&120&60&59&59&56&60\\
\text{I}\Delta'_2&&&&&&&&&&&&&&&&64&57&60&60&56&64\\
\text{I}\Delta'_3&&&&&&&&&&&&&&&&&17&29&29&29&24\\
\text{I}\Delta'_4&&&&&&&&&&&&&&&&&&17&32&31&28\\
\text{I}\Delta'_5&&&&&&&&&&&&&&&&&&&17&31&28\\
\text{I}\Delta'_6&&&&&&&&&&&&&&&&&&&&13&28\\
\text{I}\Delta'_7&&&&&&&&&&&&&&&&&&&&&27
\end{array}
\)
}
	\caption{Pairs of dotted $9$-configurations in $17$-configurations}
	\label{tbl:pairs9in17}
\end{table}

\subsection{Classification of $(6,M,4;3)_2$ codes with $M\geq 77$}
By Lemma~\ref{lma:17conf}, a $(6,M,4;3)_2$ code with
$M\geq 77$ is the extension of some $17$-configuration.  Given a
fixed $17$-configuration $\mathcal{S}$, we formulate the extension
problem as an integer linear program: For each plane
$E\in\gaussm{V}{3}$, a variable $x_E$ is introduced, which may take the
values $0$ and $1$.  The value $x_E = 1$ indicates $E\in
\mathcal{C}$.  Now $\sdist(\mathcal{C})\geq 4$ is equivalent to
a system of linear constraints for $x_E$:
\[
    \sum_{E\in\gaussm{V}{3} ; L\subset E} x_E \leq 1\qquad\text{for all lines }L\in\gaussm{V}{2}\text{.}
\]
The fact $r(P) \leq 9$ for all points $P\in\gaussm{V}{1}$ yields the further system of linear constraints
\[
	\sum_{E\in\gaussm{V}{3} ; P\subset E} x_E \leq 9\qquad\text{for all points }P\in\gaussm{V}{1}\text{.}
\]
As argued in the introduction, $\mathcal{C}$ is a $(v,M,d;k)_q$
constant-dimension code if and only if $\mathcal{C}^\perp = \{E^\perp \mid E\in
\mathcal{C}\}$ is a $(v,M,d;v-k)_q$ constant-dimension code.
In our case, if $\mathcal{C}$ is a $(6,M,4;3)_2$ code then so is
$\mathcal{C}^\perp$.
This allows us to add the dualized constraints
\begin{align*}
	\sum_{E\in\gaussm{V}{3} ; E\subset H} x_E & \leq 9 & & \text{for all hyperplanes }H\in\gaussm{V}{5}\text{ and} \\
	\sum_{E\in\gaussm{V}{3} ; E\subset S} x_E & \leq 1 & & \text{for all solids }S\in\gaussm{V}{4}\text{.}
\end{align*}
The starting configuration $\mathcal{S}$ is prescribed by adding the
constraints
\[
    x_E = 1\qquad\text{for all }E\in\mathcal{S}\text{.}
\]
Under these
restrictions, the goal is to maximize the objective function
\[
    \sum_{E\in\gaussm{V}{3}} x_E\text{.}
    \]
To avoid unnecessary computational
branches, we add another inequality forcing this value to be $\geq
77$.

For each of the $12770$ isomorphism types of $17$-configurations
$\mathcal{S}$, this integer linear problem was fed into the ILP-solver
CPLEX running on a standard PC.  The computation time per case varied
a lot, on average a single case took about 10 minutes.  In $393$ of
the $12770$ cases, the starting configuration $\mathcal{S}$ could be
extended to a code of size $77$, and it turned out that the maximum
possible size of $\mathcal{C}$ is indeed $77$.  After filtering out
isomorphic copies, we ended up with $5$ isomorphism classes of
$(6,77,4;3)_2$ constant-dimension codes.  This proves
Theorem~\ref{thm:main}.

\subsection{Analysis of the results}\label{ssec:analysis}
The five isomorphism types will be denoted by A, B, C, D and E.
We investigated the structure of those codes by computer.
The results are discussed in this section.

If $\mathcal{C}\cong\mathcal{C}^\perp$ or, equivalently, $\mathcal{C}$
is invariant under a correlation of $\PG(5,2)$, we will call
$\mathcal{C}$ \emph{self-dual}.\footnote{In general, the map
  $\subspaces(V)\to\subspaces(V)$, $U\mapsto U^\perp$ defines a
  particular correlation (lattice anti-isomorphism) of $\PG(V)$, which
  depends on the choice of the symmetric bilinear form. Composing
  $U\mapsto U^\perp$ with all collineations of $\PG(V)$ (on either
  side) yields all correlations of $\PG(V)$ and proves equivalence of
  the two conditions. Thus self-duality is a purely geometric concept
  and does not depend on the particular bilinear form chosen.} The
five isomorphism types fall into three self-dual cases (A, B, C) and a
single pair of dual codes (D, E).

There is a remarkable property shared by all five types:
\begin{corollary}
\label{cor:lightplane}
Let $\mathcal{C}$ be a $(6,77,4;3)_2$ constant-dimension code.
There is a unique plane $S\in\gaussm{V}{3}$ such that all points on
$S$ have degree $\leq 6$ and all other points have degree $\geq 8$.
\end{corollary}

For Type~B, the plane $S$ of Corollary~\ref{cor:lightplane} is a
codeword, for all other types it is not. 
An overview of the structure is given in Table~\ref{tbl:77er}.
The column ``pos. of $S$'' gives the frequencies of $\dim(E\cap S)$
when $E$ runs through the codewords of $\mathcal{C}$. 
For example, the entry $0^{48} 1^{28} 3$ for Type~B means that there
are $48$ codewords disjoint from $S$, $28$ codewords intersecting $S$
in a point and the single codeword identical to $S$. 

\begin{table}
\centering
\resizebox{\textwidth}{!}{
\(
\begin{array}{c|cccccc}
& \#\Aut & \text{duality} & \text{degree dist.} & \text{pos. of }S & \text{types of $9$-conf.} & \#\text{$17$-conf.}\\
\hline
\text{A} & 168 & \text{self-dual} & 5^7 9^{56}         & 0^{56}1^{14}2^7 & \text{X}^{28} \text{E}^{28} & 1428 \\
\text{B} &  48 & \text{self-dual} & 5^7 9^{56}         & 0^{48}1^{28}3   & \text{E}^{40} (\text{I}\Delta')^{16} & 1428 \\
\text{C} &   2 & \text{self-dual} & 5^7 9^{56}         & 0^{48}1^{26}2^3 & \text{E}^{34} (\text{I}\Delta)^6 (\text{I}\Delta')^{16} & 1416 \\
\text{D} &   2 & \text{dual of E} & 5^3 6^4 8^4 9^{52} & 0^{48}1^{24}2^5 & \text{E}^{24} (\text{I}\Delta)^9 (\text{I}\Delta')^{19} & 1222 \\
\text{E} &   2 & \text{dual of D} & 5^3 6^4 8^4 9^{52} & 0^{48}1^{24}2^5 & \text{E}^{24} (\text{I}\Delta)^9 (\text{I}\Delta')^{19} & 1222 \\
\end{array}
\)
}
\caption{Types of $(6,77,4;3)_2$ constant-dimension codes}
\label{tbl:77er}
\end{table}

For a double check of our classification, we collected all
$17$-con\-fi\-gu\-ra\-tions which appear as a substructure of the five codes,
see Table~\ref{tbl:77er} for the numbers.  After filtering out
isomorphic copies, we ended up with the same $393$ types that we
already saw as the ones among the $12770$ $17$-configurations which
are extendible to a $(6,77,4;3)_2$ code.

\section{Computer-Free Constructions}\label{sec:nocomp}
We have seen in Section~\ref{sec:comp} that the best $(6,M,4;3)_2$
subspace codes contain large subcodes (of sizes $56$ or $48$; cf.\
Table~\ref{tbl:77er})
disjoint from a fixed plane of $\PG(5,2)$. Since the latter are easy to
construct---for example, large subcodes of binary $(3,3,2)$ LMRD codes
have this property---, it is reasonable to take such codes as the
basis of the construction and try to enlarge them as far as possible.

We start by outlining the main ideas involved in this kind of
construction, which eventually leads to 
a computer-free construction of a $(6,77,4;3)_2$ subspace code of
Type~A.  For this we assume $q=2$, for simplicity. Subsequently we
will develop the approach for general $q$.

First consider a $(6,64,4;3)_2$ LMRD code $\mathcal{L}$ obtained, for
example, by lifting a matrix version of the binary $(3,3,2)$ Gabidulin
code. By Lemma~\ref{lma:lmrd}, the $64$ planes in $\mathcal{L}$ cover
the $7\cdot 64=448$ lines of $\PG(5,2)$ disjoint from the special
plane $S=\langle\uvek_4,\uvek_5,\uvek_6\rangle$, so that no plane
meeting $S$ in a point (and hence containing $4$ lines disjoint from
$S$) can be added to $\mathcal{L}$ without decreasing the subspace
distance. Therefore, in order to overcome the restriction
$\#\mathcal{C}\leq 71$ for subspace codes $\mathcal{C}$ containing
$\mathcal{L}$, we must remove some planes from $\mathcal{L}$,
resulting in a proper subcode
$\mathcal{L}_0\subset\mathcal{L}$. Removing a subset
$\mathcal{L}_1\subseteq\mathcal{L}$ with $\#\mathcal{L}_1=M_0$ will
``free'' $7M_0$ lines, i.e.\ lines disjoint from $S$ that are no
longer covered by
$\mathcal{L}_0=\mathcal{L}\setminus\mathcal{L}_1$. If $M_0=4m_0$ is a
multiple of $4$, it may be possible to rearrange the $28M_0=4\cdot
7M_0$ free lines into $7M_0$ ``new'' planes meeting $S$ in a point
(each new plane containing $4$ free lines), and such that the set
$\mathcal{N}$ of new planes has subspace distance $4$ (equivalently,
covers no line through a point of $S$ twice). The planes in
$\mathcal{N}$ can then be added, resulting in a $(6,M,4;3)_2$ subspace
code $\mathcal{C}=\mathcal{L}_0\cup\mathcal{N}$ of size
$M=64+3m_0>64$. A detailed discussion (following below) will show that
this construction with $M_0=8$ is indeed feasible and, even more, $7$
further planes meeting $S$ in a line can be added to $\mathcal{C}$
without decreasing the subspace distance. This yields an optimal
$(6,77,4;3)_2$ subspace code of Type~A. Moreover, the construction
generalizes to arbitrary $q$, producing a $(6,q^6+2q^2+2q+1,4;3)_q$
subspace code.

\subsection{Removing Subspaces from MRD Codes}\label{ssec:remove}
Let $\mathcal{L}$ be a $(6,q^6,4;3)_q$ LMRD code, arising from a
$q$-ary $(3,3,2)$ MRD code $\mathcal{A}$; cf.\
Section~\ref{ssec:lmrd}.  Then, for any hyperplane $H$ of $\PG(5,q)$
containing $S=\langle\uvek_4,\uvek_5,\uvek_6\rangle$, the corresponding
hyperplane section $\mathcal{L}\cap H=\{E\cap H;E\in\mathcal{L}\}$
consists of all $q^6$ lines in $H\setminus S$, by
Lemma~\ref{lma:lmrd}. Writing $\cm(H)=\left(
    \begin{smallmatrix}
      \mat{Z}&\mat{0}\\\mat{0}&\imat_3
    \end{smallmatrix}
  \right)$, where $\mat{Z}\in\F_q^{2\times 3}$ is the canonical matrix
  associated with $H$ (viewed as a line in $\PG(5,q)/S$), these $q^6$
  lines are $\langle(\mat{Z}|\mat{ZA})\rangle$ with
  $\mat{A}\in\mathcal{A}$. This simplifies to
  $\langle(\mat{Z}|\mat{B})\rangle$ with $\mat{B}\in\F_q^{2\times 3}$
  arbitrary.

  Our first goal in this section is to determine which subsets
  $\mathcal{R}\subset\mathcal{A}$ (``removable subsets'') of size
  $q^2$ have the property that the corresponding hyperplane section
  $\mathcal{L}_1\cap H$,
  $\mathcal{L}_1=\left\{\langle(\imat_3|\mat{A})\rangle;
    \mat{A}\in\mathcal{R}\right\}$, consists of the $q^2$ lines
  disjoint from $S$ in a new plane (a plane meeting $S$ in a single
  point). Assuming that $\mathcal{A}$ is linear over $\F_q$ will
  simplify the characterization of removable subsets.

\begin{lemma}
  \label{lma:crucial}
  Suppose $\mathcal{A}$ is a $q$-ary linear $(3,3,2)$ MRD code, and
  $H$ is a hyperplane of $\PG(5,q)$ containing $S$ with $\cm(H)=\left(
    \begin{smallmatrix}
      \mat{Z}&\mat{0}\\\mat{0}&\imat_3
    \end{smallmatrix}
  \right)$. For $\mathcal{R}\subseteq\mathcal{A}$ the following are equivalent:
\begin{enumerate}[(i)]
\item The $H$-section $\mathcal{L}_1\cap H$
  corresponding to $\mathcal{R}$ consists of
  the $q^2$ lines disjoint from $S$ in a new plane $N$.
\item $\mat{Z}\mathcal{R}=\{\mat{ZA};\mat{A}\in\mathcal{R}\}$ is a
  line in $\LAG(2,3,q)$.
\item $\mathcal{R}=\mat{A}_0+\mathcal{D}$ for some
  $\mat{A}_0\in\mathcal{A}$ and some $2$-dimensional
  $\F_q$-subspace $\mathcal{D}$ of $\mathcal{A}$ with the following
  properties: $\mathcal{D}$ has constant rank $2$, and the
  ($1$-dimensional) left kernels of the nonzero members of
  $\mathcal{D}$ generate the row space
  $\langle\mat{Z}\rangle$.\footnote{The term ``constant-rank'' has its
  usual meaning in Matrix Theory, imposing the same rank on all
  nonzero members of the subspace.} 
\end{enumerate}
If these conditions are satisfied then the new plane $N$ has $\cm(N)=\left(
    \begin{smallmatrix}
      \mat{Z}&\mat{ZA}_0\\\mat{0}&\vek{s}
    \end{smallmatrix}
\right)$, where $\vek{s}\in\F_q^3$ is a generator of the
common $1$-dimensional row space of the nonzero matrices in
$\mat{Z}\mathcal{D}$. Moreover, $\F_q\vek{s}=N\cap S$ is the point at
infinity of the line $\mat{Z}\mathcal{R}$ in (ii).
\end{lemma}
\begin{proof}
  (i)$\iff$(ii): The lines $L$ in the $H$-section have canonical matrices
  $\cm(L)
  =(\mat{Z}|\mat{ZA})$,
  $\mat{A}\in\mathcal{R}$. 
  Lemma~\ref{lma:lag} (with an obvious modification) yields that
  $\mat{B}\mapsto\langle(\mat{Z}|\mat{B})\rangle$ maps the lines of $\LAG(2,3,q)$ to
  the planes of $H$ intersecting $S$ in a point. Hence
  $\mat{Z}\mathcal{R}$ is a line in $\LAG(2,3,q)$ iff the $q^2$ lines
  in the $H$-section are incident with a new plane $N$.

  (ii)$\iff$(iii): We may assume
  $\mat{0}\in\mathcal{R}=\mathcal{D}$. The point set
  $\mathcal{U}=\mat{Z}\mathcal{D}$ is a line in $\LAG(2,3,q)$ iff the
  $q^2-1$ nonzero matrices in $\mathcal{U}$ have a common
  $1$-dimensional row space, say $\F_q\vek{s}$, and account for all such
  matrices. If this is the case,
  then $\mathcal{D}$ must be an $\F_q$-subspace of $\mathcal{A}$
  (since $\mathcal{U}$ is an $\F_q$-subspace of $\F_q^{2\times 3}$ and
  $\mathcal{A}\to\F_q^{2\times 3}$, $\mat{A}\mapsto\mat{ZA}$ is
  bijective.\footnote{Here we need the assumption that $\mathcal{A}$ is
    linear.}) Further, $\mathcal{D}$ must have constant rank $2$ (since
  $\mathcal{D}\subseteq\mathcal{A}$ forces $\rank(\mat{A})\geq 2$ for
  all $\mat{A}\in\mathcal{D}$), and the left kernels of the nonzero
  matrices in $\mathcal{D}$ must generate 
  $\langle\mat{Z}\rangle$ (since the left kernels of the nonzero
  matrices in $\mathcal{U}=\mat{Z}\mathcal{D}$ account for all
  $1$-dimensional subspaces of $\F_q^2$). Conversely,
  suppose that $\mathcal{D}$ satisfies these conditions. Then all
  nonzero matrices in $\mathcal{U}$ have rank $1$, and we must show
  that they have the same row space. Let $\mat{A}_1,\mat{A}_2$ be a
  basis of $\mathcal{D}$ and
  $\vek{z}_1,\vek{z}_2\in\F_q^3\setminus\{\vek{0}\}$ with
  $\vek{z}_1\mat{A}_1=\vek{z}_2\mat{A}_2=\vek{0}$. Then $\vek{z}_1$,
  $\vek{z}_2$ span $\langle\mat{Z}\rangle$ (otherwise all matrices in
  $\mathcal{D}$ would have kernel $\F_q\vek{z}_1=\F_q\vek{z}_2$), and hence
  there exist $\lambda,\mu\in\F_q$, $(\lambda,\mu)\neq(0,0)$, such that
  $(\lambda\vek{z}_1+\mu\vek{z}_2)(\mat{A}_1+\mat{A}_2)=\vek{0}$. Expanding,
  we find $\lambda\vek{z}_1\mat{A}_2=-\mu\vek{z}_2\mat{A}_1$, implying
  that $\mat{ZA}_1$ and $\mat{ZA}_2$ have the same row space
  $\F_q(\vek{z}_1\mat{A}_2)=\F_q(\vek{z}_2\mat{A}_1)=\F_q\vek{s}$, say. Since
  $\mathcal{U}=\langle\mat{ZA}_1,\mat{ZA}_2\rangle$, the other
  nonzero matrices in $\mathcal{U}$ must have row space $\F_q\vek{s}$
  as well.

  The remaining assertions are then easy consequences.
\end{proof}

\begin{remark}
  \label{rmk:crucial}
  The conditions imposed on $\mathcal{D}$ in
  Lemma~\ref{lma:crucial}(iii) imply that the left kernels of the nonzero
  matrices in $\mathcal{D}$ form the set of $1$-dimensional subspaces
  of a $2$-dimensional subspace of $\F_q^3$ (viz.,
  $\langle\mat{Z}\rangle$). Two matrices $\mat{A}_1,\mat{A}_2\in\mathcal{A}$
  generate a $2$-dimensional constant-rank-$2$ subspace 
  with this property iff
  $\rank(\mat{A}_1)=\rank(\mat{A}_2)=2$, the left kernels
  $K_1=\F_q\vek{z}_1$, $K_2=\F_q\vek{z}_2$ of
  $\mat{A}_1$ resp.\ $\mat{A}_2$ are distinct, and
  $\vek{z}_1\mat{A}_2$, $\vek{z}_2\mat{A}_1$ are linearly
  dependent.\footnote{This is clear from the proof of the lemma.}
\end{remark}

Since the maps $\mathcal{A}\to\F_q^{2\times 3}$,
$\mat{A}\mapsto\mat{ZA}$
are bijections, we note
the following consequence of Lemma~\ref{rmk:crucial}. For each
$2$-dimensional subspace $Z$ of $\F_q^3$ (representing a
hyperplane $H$ in $\PG(5,q)$ as described above, with
$\mat{Z}=\cm(Z)$) and each $1$-dimensional subspace 
$P$ of $\F_q^3$ (representing a point of $S$, after padding the
coordinate vector $\vek{s}=\cm(P)$ with three zeros), there exists
precisely one $2$-dimensional 
subspace $\mathcal{D}=\mathcal{D}(Z,P)$ of $\mathcal{A}$ with the
properties in Lemma~\ref{lma:crucial}(iii). The subspace $\mathcal{D}$
consists of all $\mat{A}\in\mathcal{A}$ with $\langle\mat{ZA}\rangle=P$.

\begin{example}
  \label{ex:crucial}
  We determine the subspaces $\mathcal{D}(Z,P)$ for the $q$-ary
  $(3,3,2)$ Gabidulin code $\mathcal{G}$; cf.\ Example~\ref{ex:332q}.
  Working in a coordinate-independent manner, suppose 
  $Z=\langle a,b\rangle$ with $a,b\in\F_{q^3}^\times$ linearly independent
  over $\F_q$ and $P=\langle c\rangle$ with $c\in\F_{q^3}^\times$. Since
  $x\mapsto ux^q-u^qx$,
  $u\in\F_{q^3}^\times$, has kernel $\F_qu$, the maps
  \begin{equation*}
    f(x)=\frac{c(ax^q-a^qx)}{ab^q-a^qb},\quad g(x)=\frac{c(bx^q-b^qx)}{ba^q-b^qa}
  \end{equation*}
  are well-defined, have rank $2$, and satisfy $f(a)=g(b)=0$,
  $f(b)=g(a)=c$. Hence $\mathcal{D}(Z,P)=\langle f,g\rangle$.
  We may also write
  \[
      \mathcal{D}\bigl(Z,\F_q(ab^q-a^qb)\bigr)=\{ux^q-u^qx;u\in Z\}\text{,}
  \]
  making the linear dependence on $Z$ more visible. Scaling by a
  nonzero constant in $\F_{q^3}^\times$ then 
  yields the general $\mathcal{D}(Z,P)$.
\end{example}

It is obvious from Lemma~\ref{lma:crucial} that
$\mathcal{D}(Z_1,P_1)\neq\mathcal{D}(Z_2,P_2)$ whenever
$(Z_1,P_1)\neq(Z_2,P_2)$. Hence a single coset
$\mathcal{R}=\mat{A}_0+\mathcal{D}\subset\mathcal{A}$ leads only to a
single new plane in one particular hyperplane section determined by
$\mathcal{D}$, and therefore Lemma~\ref{lma:crucial} cannot be
directly applied to yield $(6,M,4;3)_q$ subspace codes larger than
LMRD codes. In order to achieve $\#\mathcal{N}>\#\mathcal{R}$,
we should instead look for larger sets $\mathcal{R}$ having the
property that the lines in the corresponding $q^2+q+1$
hyperplane sections can be simultaneously
arranged into new planes. This requires $\mathcal{R}$ to be a union of
cosets of spaces $\mathcal{D}(Z,P)$ simultaneously for all $Z$, for
example a subspace containing a space $\mathcal{D}(Z,P)$ for each
$Z$. Further we require that the corresponding points $P$ are
different for different choices of $Z$, excluding ``unwanted''
multiple covers of lines through $P$ by new planes in different
hyperplanes $H\supset S$.
\begin{lemma}
  \label{lma:simul}
  Let $\mathcal{R}$ be a $t$-dimensional
  $\F_q$-subspace of a $q$-ary linear $(3,3,2)$ MRD
  code $\mathcal{A}$, having the following properties:
  \begin{enumerate}[(i)]
  \item For each
  $2$-dimensional subspace $Z$ of $\F_q^3$ there exists a
  $1$-di\-men\-sio\-nal subspace $P=Z'$ of $\F_q^3$ such that
  $\mathcal{D}(Z,P)\subseteq\mathcal{R}$.
  \item The map $Z\mapsto Z'$ defines a bijection from
    $2$-dimensional subspaces to $1$-dimensional subspaces of $\F_q^3$. 
  \item $\rank\left(
      \begin{smallmatrix}
        \mat{ZA}_1-\mat{ZA}_2\\
        \vek{s}
      \end{smallmatrix}
\right)=3$ whenever $\mat{A}_1$, $\mat{A}_2$ are in different cosets
of $\mathcal{D}(Z,P)$ in $\mathcal{R}$, where $P=Z'$,
    $\mat{Z}=\cm(Z)$, and $\vek{s}=\cm(P)$.
  \end{enumerate}
  Then the $(q^2+q+1)q^t$ lines covered by the planes in
  $\mathcal{L}$ corresponding to $\mathcal{R}$ can be
  rearranged into $(q^2+q+1)q^{t-2}$ new planes meeting $S$ in a point
  and such that the set $\mathcal{N}$ of new planes has minimum
  subspace distance $4$. Consequently, the remaining $q^6-q^t$ planes in
  $\mathcal{L}$ and the new planes in $\mathcal{N}$ constitute a
  $(6,q^6+q^{t-1}+q^{t-2},4;3)_q$ subspace code.
\end{lemma}
\begin{proof}
  As before let
  $\mathcal{L}_1$ denote the set of $q^t$ planes of the form
  $\langle(\imat_3|\mat{A})\rangle$, $\mat{A}\in\mathcal{R}$.
  By (i), the $q^t$ lines in any of the $q^2+q+1$ hyperplane sections
  $\mathcal{L}_1\cap H$ are partitioned into $q^{t-2}$ new planes
  meeting $S$ in the same point $P=Z'$. Condition~(ii) ensures that
  new planes in different hyperplanes have no line in common
  and hence subspace distance $\geq 4$. Finally, if $N$, $N'$ are
  distinct new planes in the same hyperplane section then
  $\cm(N)=\left(\begin{smallmatrix}
      \mat{Z}&\mat{ZA}\\
      \vek{0}&\vek{s}
    \end{smallmatrix}\right)$, $\cm(N')=\left(\begin{smallmatrix}
      \mat{Z}&\mat{ZA'}\\
      \vek{0}&\vek{s}
    \end{smallmatrix}\right)$ for some $\mat{A}$,
  $\mat{A}'\in\mathcal{R}$ with
  $\mat{A}+\mathcal{D}\neq\mat{A}'+\mathcal{D}$ where
  $\mathcal{D}=\mathcal{D}(Z,P)$. Condition~(iii) is equivalent to
  $N+N'=H$, i.e.\ $N\cap N'=P$ or $\sdist(N,N')=4$. 
\end{proof}
Condition~(iii) in Lemma~\ref{lma:simul} is also equivalent to the
requirement that
the set $\mathcal{N}_H$ of $q^{t-2}$ new planes in any hyperplane
$H\supset S$ must form a partial spread
in the quotient geometry $\PG(H/P)\cong\PG(3,q)$. This implies $t\leq 4$
with equality iff $\mathcal{N}_H\cup\{S\}$ forms a spread in $\PG(H/P)$. As
$t\leq 2$ is impossible (cf. the remarks before Lemma~\ref{lma:simul}),
our focus from now on will be on the case $t=3$. 
\begin{example}
  \label{ex:simul}[Continuation of Example~\ref{ex:crucial}]
  Setting $\mathcal{R}=\{ux^q-u^qx;u\in\F_{q^3}\}$ and 
  $Z'=P=\F_q(ab^q-a^qb)$ for a $2$-dimensional
  subspace $Z=\langle a,b\rangle$ of $\F_{q^3}$, we have
  $\mathcal{D}(Z,P)\subset\mathcal{R}$ for any $Z$. 
  Using $ab^q-a^qb=a^{q+1}\bigl((a^{-1}b)^q-(a^{-1}b)\bigr)$, the
  additive version of Hilbert's
  Theorem~90 \cite[Satz~90]{hilbert97}
  and $\gcd(q^2+q+1,q+1)=1$, we see that $Z\mapsto Z'$ is
  bijective.\footnote{In fact $Z\mapsto Z'$ defines a correlation of
    $\PG(\F_{q^3})\cong\PG(2,q)$: The image of
    the line pencil through $\F_qa$ is the set of points on the line
    with equation $\trace_{\F_{q^3}/\F_q}(xa^{-q-1})=0$.}
  Condition~(iii) of Lemma~\ref{lma:simul} is equivalent to $\langle
  ca^q-c^qa,cb^q-c^qb,ab^q-a^qb\rangle=\F_{q^3}$ whenever $\langle
  a,b,c\rangle=\F_{q^3}$. This is in fact true, as we now show: $a,b,c$ form a
  basis of $\F_{q^3}/\F_q$ iff
  \begin{equation*}
    \begin{vmatrix}
      a&b&c\\
      a^q&b^q&c^q\\
      a^{q^2}&b^{q^2}&c^{q^2}
    \end{vmatrix}
    \neq 0.
  \end{equation*}
  The adjoint determinant is
  \begin{equation*}
    \begin{vmatrix}
      A^q&A^{q^2}&A\\
      B^q&B^{q^2}&B\\
      C^q&C^{q^2}&C
    \end{vmatrix}
  \end{equation*}
  with $A=bc^q-b^qc$, $B=ca^q-c^qa$, $C=ab^q-a^qb$. By the same token,
  the adjoint determinant is $\neq 0$ iff $A$, $B$, $C$ form a basis
  of $\F_{q^3}/\F_q$. From this our claim follows.

  Thus $\mathcal{R}$ satisfies all conditions of
  Lemma~\ref{lma:simul} and gives rise to a $(6,q^6+q^2+q,4;3)_q$
  subspace code $\mathcal{C}$
  consisting of $q^6-q^3$ ``old'' planes disjoint from $S$ and
  $q^3+q^2+q$ ``new'' planes meeting $S$ in a point, $q$ of them
  passing through any of the $q^2+q+1$ points in $S$.  

  In the coordinate-free model introduced in Section~\ref{ssec:lmrd},
  $\mathcal{C}$ consists of the $q^6-q^3$ planes
  $G(a_0,a_1)=\bigl\{(x,a_1x^q-a_0x);x\in\F_{q^3}\bigr\}$ with
  $a_0,a_1\in\F_{q^3}$, $a_0\neq a_1^q$ and the $q(q^2+q+1)$ planes
  $N(a,b,c)=\bigl\{(x,cx^q-c^qx+y(ab^q-a^qb);x\in Z,y\in\F_q\bigr\}$
  with $Z=\langle a,b\rangle$ any $2$-dimensional $\F_q$-subspace of
  $\F_{q^3}$ and $c\in\F_{q^3}/Z$.
\end{example}

\subsection{Construction~A and the proof of
  Theorem~\ref{thm:allq}}\label{ssec:A}
In this section we complete the construction of an optimal
$(6,77,4;3)_2$ subspace code of Type~A in Table~\ref{tbl:77er}. This
will be done by adding $7$ further planes to the $(6,70,4;3)_2$ code
$\mathcal{C}$ of Example~\ref{ex:simul}. Since
$\mathcal{C}=\mathcal{L}_0\oplus\mathcal{N}$ already covers all 
  lines disjoint from $S$, these planes must meet $S$ in a
  line.\footnote{The plane $S$ itself may also be added to $\mathcal{C}$
    without decreasing the subspace distance, resulting in a maximal
    $(6,73,4;3)_2$ subspace code.}
  Hence the augmented $(6,77,4;3)_2$ subspace code will contain
  precisely $56$ planes disjoint from $S$ and thus be of Type~A.

In fact there is nothing special with the case $q=2$ up to this point,
and for all $q$ we can extend the code of Example~\ref{ex:simul} by
adding $q^2+q+1$
further planes meeting $S$ in a line. This is the subject of the next
lemma, which thereby completes the proof of Theorem~\ref{thm:allq}.

\begin{lemma}
  \label{lma:A}
  The subspace code $\mathcal{C}$ from Example~\ref{ex:simul} can be
  extended to a $(6,q^6+2q^2+2q+1,4;3)_q$ subspace code
  $\overline{\mathcal{C}}$.
\end{lemma}
\begin{proof}
  Our task is to add
  $q^2+q+1$ further planes to $\mathcal{C}$, one plane $E$ with
  $E\cap S=L$ for each line $L$ in $S$. These planes should
  cover points outside $S$ at most once (ensuring that no line through
  a point of $S$ is covered twice), and we must avoid adding
  planes that contain a line already covered by $\mathcal{N}$.

  First we will determine the points outside $S$ covered by the planes
  in $\mathcal{N}$. These are precisely the points covered by the
  ``replaced'' planes in $\mathcal{L}_1$. Viewed as points in
  $\PG(\F_{q^3}\times\F_{q^3})$, they have the form $(x,ux^q-u^qx)$
  with $u\in\F_{q^3}$, $x\in\F_{q^3}^\times$. Since
  $ux^q-u^qx=-x^{q+1}\bigl((ux^{-1})^q-ux^{-1}\bigr)$, the points
  covered by $\mathcal{L}_1$ are the $q^2(q^2+q+1)$ points
  $\F_q(x,x^{q+1}v)$ with $x,v\in\F_{q^3}$, $x\neq 0$ and
  $\trace_{\F_{q^3}/\F_q}(v)=0$, each such point being covered exactly
  $q$ times.\footnote{From this it follows that any plane in
    $\mathcal{L}_1$ intersects $(q-1)(q^2+q+1)=q^3-1$ further planes
    in $\mathcal{L}_1$, i.e., the $q^3$ planes in $\mathcal{L}_1$
    mutually intersect in a point. This can also be concluded from the
    fact that the matrix space $\mathcal{R}$ has constant rank $2$.}
  The $q^5+q^4+q^3-(q^4+q^3+q^2)=q^5-q^2$ points outside $S$ and not
  covered by $\mathcal{L}_1$ are those of the form $\F_q(x,x^{q+1}v)$
  with $\trace_{\F_{q^3}/\F_q}(v)\neq 0$.

  Using this representation we can proceed as follows: We choose
  $v_0\in\F_{q^3}$ with $\trace_{\F_{q^3}/\F_q}(v_0)\neq 0$ and
  connect the point $\F_q(x,x^{q+1}v_0)$ to the line in
  $S=\{0\}\times\F_{q^3}=\bigl\{(0,y);y\in\F_{q^3}\bigr\}$ with
  equation $\trace_{\F_{q^3}/\F_q}(yx^{-q-1})=0$. This gives $q^2+q+1$
  planes $E(x)$, $x\in\F_{q^3}^\times/\F_q^\times$, which intersect
  $S$ in $q^2+q+1$ different lines (since $\F_qx\mapsto\F_qx^{q+1}$
  permutes the points of $\PG(\F_{q^3})$) and cover no point already
  covered by $\mathcal{L}_1$ (since the points in $E(x)$ have the form
  $\F_q\bigl(x,x^{q+1}(v_0+v)\bigr)$ with
  $\trace_{\F_{q^3}/\F_q}(v)=0$ and hence
  $\trace_{\F_{q^3}/\F_q}(v_0+v)\neq 0$).\footnote{It is unfortunate
    that we cannot use lines in $S$ multiple times, since we are
    obviously able to pack the $q^5-q^2$ points outside $S$ and not
    covered by $\mathcal{L}_1$ using $q(q^2+q+1)$ such planes, $q$
    through every line of $S$, by choosing $q$ different values for
    $\trace_{\F_{q^3}/\F_q}(v_0)$.}  Clearly the latter implies that
  $E(x)$ has no line in common with a plane in $\mathcal{N}$. Finally,
  projection onto the first coordinate (in $\F_{q^3}$) shows that
  distinct planes $E(x)$ and $E(x')$ do not have points outside $S$ in
  common and hence intersect in a single point $P\in S$. In all we
  have now shown that the extended subspace code
  $\overline{\mathcal{C}}
  =\mathcal{C}\cup\bigl\{E(x);x\in\F_{q^3}^\times/\F_q^\times\bigr\}$
  has the required parameters.
\end{proof}

\section{Conclusion}\label{sec:concl}
We conclude the paper with a list of open questions
arising from the present work.

\begin{oprobs}
  \begin{enumerate}
  \item Determine the maximum sizes $\smax_2(6,d)$ of binary
    ``mixed-dimension'' subspace codes with packet length (ambient
    space dimension) $6$ and minimum subspace distance~$d$, $1\leq
    d\leq 6$.\footnote{Some of the values $\smax_2(6,d)$ are
      already known. For example, $\smax_2(6,1)=2825$ (the total number of
      subspaces of $\F_2^6$) and $\smax_2(6,5)= \smax_2(6,6)=9$ (as
      is easily verified).}
  \item Give computer-free constructions of $(6,77,4;3)_2$ subspace
    codes of Types B, C, D and E.
  \item Find a computer-free proof of the upper bound
    $\smax_2(6,4;3)\leq 77$. A first step in this direction would be the
    proof that $(6,M,4;3)_2$ subspaces codes with $M\geq 77$ determine
    a distinguished plane $S$ with the property in
    Corollary~\ref{cor:lightplane}. 
  \item For $q>2$, the best known bounds on $\smax_q(6,d)$ are provided by Theorem~\ref{thm:allq} and Lemma~\ref{lma:bound6}:
  \[
      q^6 + 2q^2 + 2q + 1 \leq \smax_q(6,4;3) \leq q^6 + 2q^3 + 1
  \]
  Reduce the remaining gap of size $2(q^3 - q^2 - q)$ by improving the lower or the upper bound.
  \item Generalize Construction~A to packet lengths $v>6$ and/or
    constant dimensions $k>3$.
  \item Prove or disprove $\smax_2(7,4;3)=381$ (now the smallest open binary
    constant-dimension case), thereby resolving the existence question
    for the $2$-analog of the Fano plane.
  \end{enumerate}
\end{oprobs}

\section*{Acknowledgements.} The authors wish to thank Thomas Feulner
for providing us with his canonization algorithm from~\cite{feulner13} 
and the two reviewers for valuable comments and corrections.

\end{document}